\documentclass[12pt]{amsart}
 \usepackage{mathtools}
\mathtoolsset{showonlyrefs,showmanualtags}

\usepackage[backrefs]{amsrefs}

 \usepackage[top=1.5in, bottom=1.5in, left=1.25in, right=1.25in]	{geometry}
\usepackage[all]{xy}
\usepackage{amsmath}
\usepackage{amssymb}
\usepackage{amsthm}
\usepackage{amscd}

\newcommand{\RR}{\mathbb  R}

\newcommand{\Z}{\mathbb  Z}

\numberwithin{equation}{section}

\newtheorem{theorem}[equation]{Theorem}

\newtheorem{proposition}[equation]{Proposition}

\newtheorem{lemma}[equation]{Lemma}

\newtheorem{remark}[equation]{Remark}

\usepackage{hyperref} 
\hypersetup{
    colorlinks=true,       
    linkcolor=blue,          
    citecolor=magenta,        
    filecolor=magenta,      
    urlcolor=cyan           
}

\begin{document}
\title{Discrete Fractional Integration Operators Along the Primes}

\author{Ben Krause}
\address{
Department of Mathematics,
Caltech \\
Pasadena, CA 91125}
\email{benkrause2323@gmail.com}
\date{\today}

\begin{abstract}
We prove that the discrete fractional integration operators along the primes
\[ T^{\lambda}_{\mathbb{P}}f(x) := \sum_{p} \frac{f(x-p)}{p^{\lambda}} \cdot \log p \]
are bounded $\ell^p\to \ell^{p'}$ whenever $ \frac{1}{p'} < \frac{1}{p} - (1-\lambda), \ p > 1.$
Here, the sum runs only over prime $p$.
\end{abstract}

\maketitle


\section{Introduction}
The topic of this paper is discrete harmonic analysis, with a  focus on discrete analogues of fractional integral operators. While elementary arguments link the discrete operators, 
\[ J_{d,\lambda}f(x):=\sum_{0 \neq m \in \mathbb{Z}^d} \frac{f(x-m)}{|m|^{d \cdot \lambda}}\]
to their continuous analogues -- which leads to the expected range of norm estimates: $J_{d,\lambda}$ maps $\ell^{p}(\mathbb{Z}^d) \to \ell^{p'}(\mathbb{Z}^d)$ when $1/p' \leq 1/p - (1-\lambda)$ -- the problem becomes much more subtle upon the introduction of radon behavior, where the primary objects of consideration (in the one-dimensional setting) are of the form
\begin{equation}\label{e:fracmon}
I_\lambda^sf(x) := \sum_{m \geq 1} \frac{f(x-m^s)}{m^\lambda}.
\end{equation}
While it is expected that $I_{\lambda}^s$ maps $\ell^p \to \ell^{q}$ when 
\begin{itemize}
    \item $1/p  > 1-\lambda$ and
    \item $\frac{1}{p'} \leq \frac{1}{p} - \frac{1-\lambda}{s}$,
\end{itemize}
significant number-theoretic complications aside from the $s=2$ case, treated in \cite{SW00, SW02, Ob, IW}, have made it difficult to obtain these full range of exponents; see \cite{P0} for a discussion of these number-theoretic complications. Indeed, Pierce's work on fractional radon transforms \cite{P} along curves of the form (say)
\[ \sum_{0 \neq m \in \mathbb{Z}^d} \frac{f(x-m,y-Q(m))}{|m|^{d \cdot\lambda} }, \]
$Q$ a quadratic form, further established the link between the quadratic nature of the the curves in question and the ability to obtain wide range of $\ell^p$ estimates for fractional radon transforms.

In this short note, we explore the case of fractional radon transforms along the {primes},
\begin{equation}\label{e:fracP}
T^{\lambda}_{\mathbb{P}}f(x) := \sum_{p } \frac{f(x-p)}{p^{\lambda}} \cdot \log p;
\end{equation}
these fractional radon transforms do not have a quadratic nature to the operator (the sum runs over prime $p$, and the presence of the logarithm is a normalizing factor, appearing from density considerations).

Nevertheless, drawing upon the techniques of \cite{KesLac}, we prove the following theorem.

\begin{theorem}\label{t:main}
Suppose that $p > 1$ and $1/p' < 1/p - (1-\lambda)$. Then \eqref{e:fracP} maps $\ell^p \to \ell^{p'}$.
\end{theorem}
\begin{remark}
The hard restriction $\frac{1}{p'} < 1/p - (1-\lambda)$ is an artifact of the real interpolation method we use; a soft inequality is expected, and would follow from complex methods. We do not pursue this issue here.
\end{remark}

\subsection{Notation}
Here and throughout, $e(t) := e^{2\pi i t}$. We let $\mu$ and $\phi$ denote the M\"{o}bius and totient functions, respectively. A key estimate is the lower bound 
\begin{equation}\label{e:tot}
    \phi(q) \gtrsim_\epsilon q^{1-\epsilon}
\end{equation}
valid for any $\epsilon > 0$.


For $k \geq 1$ we let
\begin{equation}\label{e:prime1}
    K_k(x) := \frac{1}{2^k} \sum_{p \leq 2^k} \delta_p(x) \cdot \log p,
\end{equation}
where the sum runs over only prime $p$.

We will make use of the modified Vinogradov notation. We use $X \lesssim Y$, or $Y \gtrsim X$, to denote the estimate $X \leq CY$ for an absolute constant $C$. We use $X \approx Y$ as shorthand for $Y \lesssim X \lesssim Y$. We also make use of big-O notation: we let $O(Y )$ denote a quantity that is $\lesssim Y$. If we need $C$ to depend on a parameter, we shall indicate this by subscripts, thus for instance $X \lesssim_p Y$ denotes the estimate $X \leq C_p Y$
for some $C_p$ depending on $p$. We analogously define $O_p(Y)$.

\section{The Argument}
By an appeal to the triangle inequality, Theorem \ref{t:main} will follow from the following proposition.

\begin{proposition}\label{p:key}
For any $p > 1$,
\begin{equation}\label{e:mainest}
\| K_k*f\|_{\ell^{p'}} \lesssim_\epsilon 2^{-k \cdot (1/p - 1/p' - \epsilon)} \|f\|_{\ell^p}.
\end{equation}
\end{proposition}
Since our range of exponents is open, it suffices to prove a restricted weak-type estimate:
\begin{equation}\label{e:mainest1}
\langle K_k*\mathbf{1}_F, \mathbf{1}_G \rangle \lesssim_\epsilon
|Q| \cdot \left( \frac{|F|}{|Q|} \right)^{1 - \epsilon} \cdot \left( \frac{|G|}{|Q|} \right)^{1 - \epsilon}
\end{equation}
whenever $F, G \subset Q := Q_k := \{ |x| \lesssim 2^k\}$.

It is \eqref{e:mainest1} to which we turn.

We next recall the following multi-frequency multiplier theorems, which in turn grew out of \cite{IW}.

\subsection{A Multi-Frequency Multiplier Theorem for Ionescu-Wainger Type Multipliers}\label{s:IW}
The results of this section appear as the special one-dimensional case of \cite[Theorem 5.1]{MST1}, a refinement of \cite[Theorem 1.5]{IW} (which would also be adequate for our purposes).
\begin{theorem}[Special Case]\label{Multtheorem}
Suppose that $m(\xi)$ is an $L^p(\RR)$ multiplier with norm $A$:
\[ \|  \left( m(\xi) f(\xi) \right)^{\vee} \|_{L^p(\RR)} \leq A \| f \|_{L^p(\RR)}.\]
Let $\rho >0$ be arbitrary (for applications, we will take $0<\rho \ll_{p} 1$). Then, for every $N$, there exists an absolute constant $C_\rho > 0$ so that one may find a set of rational frequencies
\[ \left \{ \frac{a}{q} \text{ reduced} : q \leq N \right \} \subset \mathcal{U}_N \subset 
\left \{ \frac{a}{q} \text{ reduced} : q \leq C_\rho e^{N^{\rho}} \right \},\]
so that
\begin{equation}\label{IWM}
\left( \sum_{\theta \in \mathcal{U_N}} m(\alpha - \theta) \eta_N(\alpha - \theta) \hat{f}(\beta) \right)^{\vee} 
\end{equation}
has $\ell^p$ norm $\lesssim_{\rho,p} A \cdot \log N$. Here, $\eta_N$ is a smooth bump function supported in a ball centered at the origin of radius $\leq e^{-N^{2\rho}}$.
\end{theorem}
\begin{remark}
This result should be contrasted with the strongest analogous multiplier theorem for \emph{general} frequencies, which accrues a norm loss of
\[ \left( \text{Number of Frequencies} \right)^{|1/2-1/p|},\]
even in the special case when $m \in \mathcal{V}^2(\mathbb{R})$ has finite \emph{$2$-variation}, see \cite[Lemma 2.1]{C+}.
\end{remark}

\subsection{Decompositions}
Set $\mathcal{U} := \mathcal{U}_{k^{C_0}}$ for some $C_0 \gg_p 1$, where $\mathcal{U}$ is as in Theorem \ref{Multtheorem}, and choose $0 < \rho = \rho_p \ll 1$. With these choices in mind, 
decompose $f = \mathbf{1}_F$ via the Fourier transform as
\[    \hat{f} = \widehat{f_1} + \widehat{f_2},
\]
where 
\begin{equation}\label{e:flow}
\widehat{f_1}(\alpha) := \sum_{\theta \in \mathcal{U}} \eta_k(\alpha - \theta) \cdot \hat{f}(\alpha),
\end{equation}
for $\eta$ a compactly supported bump function that is one in a neighborhood of the origin, and $\eta_k(t) := \eta(2^{k^{\rho'}} t)$ for some $1 \gg \rho' \gg \rho$. 

Note that for any $p > 1$
so that 
\begin{equation}\label{e:l1comp}
\| f_1\|_{\ell^p} \lesssim \log k \cdot \|f\|_{\ell^p}.
\end{equation}

We will also decompose that Fourier transform of $K_k$
\begin{equation}\label{e:FK}
    \widehat{K_k}(\alpha) := \frac{1}{2^k} \sum_{p \leq 2^k} e( - p \alpha) \cdot \log p.
\end{equation}

To do so, we recall the following approximation result of \cite{MTZK}.

\begin{lemma}\label{l:app}
For any $A \gg 1$, there exists a $C =C(A)$ so that one may decompose $\widehat{K_k}(\alpha) = L_k'(\alpha) + \mathcal{E}_k(\alpha)$, where
\[L_k'(\alpha) := \sum_{t} L_{k,t}'(\alpha)\]
with
\begin{equation}\label{e:L}
    L_{k,t}'(\alpha) := \sum_{q \approx 2^t} \frac{\mu(q)}{\phi(q)} \sum_{(a,q) = 1} V_k(\alpha - a/q) \chi_{t}(\alpha - a/q),
\end{equation}
and $|\mathcal{E}_k| \lesssim k^{-A} $ pointwise. Here, $\chi_t(\alpha) := \chi(2^{Ct} \alpha)$ is a compactly supported bump function, and $V_k(\alpha) := \int_0^1 e(- 2^k t \cdot \alpha ) \ dt$.
\end{lemma}

For our purposes, we will replace $L_k'$ with $L_k := \sum_{t} L_{k,t}$, where
\begin{equation}\label{e:L}
    L_{k,t}(\alpha) := \sum_{q \approx 2^t} \frac{\mu(q)}{\phi(q)} \sum_{(a,q) = 1} V_k(\alpha - a/q) \varphi_{k}(\alpha - a/q),
\end{equation}
where $\varphi$ is a compactly supported bump function, and $\varphi_k(\alpha) := \varphi(2^{k(1-\epsilon)} \alpha )$.

Specifically, we have the following lemma.
\begin{lemma}
The following estimate holds:
\[ \sup_{\alpha} |L_k(\alpha) - L_{k}'(\alpha)| \lesssim 2^{-\epsilon k}.\]
\end{lemma}
\begin{proof}
It suffices to show that 
\[ |L_{k,t}(\alpha) - L_{k,t}'(\alpha)| \lesssim 2^{-(\epsilon-1) t} \cdot 2^{- \epsilon k}.\]
The key point is that if $\varphi_k(\alpha) - \chi_t(\alpha)$ does not vanish, then $2^{(\epsilon - 1)k} \lesssim |\alpha|$, so that
\[ |V_k(\alpha)| \lesssim (2^k|\alpha|)^{-1} \lesssim 2^{-\epsilon k}.\]
The result then follows from the fact that $\{ \chi_t(\cdot - a/q) : (a,q) = 1, \ q \approx 2^t \}$ are disjointly supported, taking into account the decay of the totient function, \eqref{e:tot}. 
\end{proof}

With these decompositions in hand we turn to the proof.

\begin{proof}
To prove \eqref{e:mainest1} it suffices to bound
\[ |K_k* f| \leq M_1 f + M_2f, \ f = \mathbf{1}_F\]
where for any $p > 1$
\begin{equation}\label{e:LOW}
\| M_1 f \|_{p'} \lesssim \frac{k^2}{2^{k \cdot (2/p-1)}} \cdot \|f\|_p
\end{equation}
and
\begin{equation}\label{e:HIGH}
\| M_2 f\|_2 \lesssim k^{- C} \cdot \|f \|_2,
\end{equation}
where $C$ may be adjusted to be as large as we wish.

Our decomposition is as follows:
\begin{equation}\label{e:decomp}
    K_k*f = K_k * f_1 + \left( L_k \widehat{f_2} \right)^{\vee} + \left( \mathcal{E}_k \widehat{f_2} \right)^{\vee}.
\end{equation} 
We set $M_1 f:= K_k*f_1$; by interpolating between the $\ell^1 \to \ell^\infty$ bound of $\frac{k}{2^k}$, and the trivial $\ell^2 \to \ell^2$ bound, we see that
\begin{equation}\label{e:imp0}
\| K_k*f\|_{\ell^{p'}} \lesssim \frac{k}{2^{k \cdot (2/p-1)}} \cdot \|f \|_{\ell^p}, \ p > 1
\end{equation}
which leads to the estimate \eqref{e:LOW}, taking into account \eqref{e:l1comp}.

The contribution of the term involving $\mathcal{E}_k$ is absorbed into $M_2 f$, but contributes a negligible bound, as we are free to choose $A$ in Lemma \ref{l:app} as large as we wish; it suffices to show that $L_k \widehat{f_2}$ satisfies the $\ell^2$ estimate, \eqref{e:HIGH}.

In particular, we need to estimate
\[ \sum_t \| L_{k,t} \widehat{f_2} \|_2 = \sum_{t:2^t \geq k^{C_0}} \| L_{k,t} \widehat{f_2} \|_2;\]
using the decay of the totient function, \eqref{e:tot}, a bound of $k^{\epsilon - C_0}$ is obtained, which yields the result.
\end{proof}

\typeout{get arXiv to do 4 passes: Label(s) may have changed. Rerun}

\end{document}